\documentclass[12pt]{article}

\usepackage{fullpage}
\usepackage{amssymb}
\usepackage{amsmath}
\usepackage{amsthm}
\usepackage{url}
\usepackage{hyperref}
\usepackage{tabularx}

\newtheorem{theorem}{Theorem}
\newtheorem{corollary}{Corollary}

\newtheorem{lemma}{Lemma}

\newtheorem{definition}{Definition}

\newcommand{\OEIS}[1]{\href{https://oeis.org/#1}{#1}}

\begin{document}

\title{Sums of products of binomial coefficients mod $2$ and run length transforms of sequences}

\author{Chai Wah Wu\\ IBM Research AI\\IBM T. J. Watson Research Center\\ P. O. Box 218, Yorktown Heights, New York 10598, USA\\e-mail: cwwu@us.ibm.com}
\date{October 19, 2016\\Latest revision: August 12, 2022}
\maketitle

\begin{abstract}
We study properties of functions of binomial coefficients mod 2 and derive a set of recurrence relations for
sums of products of binomial coefficients mod 2. We show that they result in sequences that are the run length transforms of well known basic sequences.
In particular, we obtain formulas for the run length transform of the positive integers, Fibonacci numbers, extended Lucas numbers and Narayana's cows sequence.
\end{abstract}

\section{Introduction}
When is the binomial coefficient even or odd, i.e. what is $\left(\begin{array}{c}n\\k\end{array}\right) \mod 2$?
It is well known that when Pascal's triangle of binomial coefficients is taken mod 2, the result has a fractal structure in the limit and corresponds to Sierpi\'{n}ski's triangle (also known as Sierpi\'{n}ski's gasket or Sierpi\'{n}ski's sieve) \cite{Stewart:sierpinski:1995,mathworld:sierpinski,LEROY201624,Mathonet2022}.

Lucas' theorem \cite{fine:lucas:1947,granville:bincoef:1997} provides a simple way to determine the binomial coefficients modulo a prime. It states that for integers $k$, $n$ and prime $p$, the following relationship holds:
\[ \left(\begin{array}{c}n\\k\end{array}\right) \equiv \prod_{i=0}^{m} \left(\begin{array}{c}n_i\\k_i\end{array}\right) \mod  p \]
where $n_i$ and $k_i$ are the digits of $n$ and $k$ in base $p$ respectively\footnote{If the lengths of the base $p$ representations of $n$ and $k$ differ, leading $0$'s are prepended to the shorter representation.}.

When $p=2$, $n_i$ and $k_i$ are the bits in the binary expansion of $n$ and $k$ and $\left(\begin{array}{c}n_i\\k_i\end{array}\right)$ is $0$ if and only if $n_i < k_i$. This implies that $\left(\begin{array}{c}n\\k\end{array}\right)$ is even if and only if  $n_i < k_i$ for some $i$.

The truth table of $n_i < k_i$ is:

\[\begin{array}{ c | c || c }
  n_i & k_i & n_i < k_i \\ \hline
  0 & 0 & 0 \\
  0 & 1 & 1 \\
  1 & 0 & 0 \\
  1 & 1 & 0 \\
\end{array}\]
and is logically equivalent to $k_i \wedge ( \neg n_i)$.  Let us consider the notation $\wedge$, $\vee$ and $\neg$ to also function as operations on integers by treating them as bitwise operations \cite{wikipedia:bitwise_operations,brent:arithmetic:2010} on the binary representation of numbers\footnote{Again, leading $0$'s are added to the binary operations $\wedge$ and $\vee$ if the operands differ in bit lengths. Furthermore, negative integers are represented in binary using the 2's complement format \cite{wikipedia:twos_complement,brent:arithmetic:2010}.}. For instance, $11 \wedge 14$ is the bitwise AND of $1011_2$ and $1110_2$ which is equal to $1010_2 = 10$.
This implies the following well-known fact \cite{mathworld:sierpinski}:

\begin{theorem}\label{thm:one}
$ \left(\begin{array}{c}n\\k\end{array}\right) \equiv 0 \mod 2$ if and only if $k \wedge ( \neg n)\neq 0$.
\end{theorem}

Incidentally, for bits $n_i$ and $k_i$, $n_i < k_i$ is logically equivalent to $\neg (k_i\Rightarrow n_i)$.
Consider $\left(\begin{array}{c}n\\k\end{array}\right)\left(\begin{array}{c}m\\r\end{array}\right)\mod 2$.  Clearly this is equivalent to

$$\left(\left(\begin{array}{c}n\\k\end{array}\right) \mod 2\right)\left(\left(\begin{array}{c}m\\r\end{array}\right)\mod 2\right)$$

 Thus $\left(\begin{array}{c}n\\k\end{array}\right)\left(\begin{array}{c}m\\r\end{array}\right) \equiv 0 \mod 2$ if and only if $k \wedge ( \neg n)\neq 0$  or $r \wedge ( \neg m)\neq 0$.  This in turn implies the following:

\begin{theorem}\label{thm:two}
\[\begin{array}{l}
\left(\begin{array}{c}n\\k\end{array}\right)\left(\begin{array}{c}m\\r\end{array}\right) \equiv 0 \mod 2 \Leftrightarrow 
(k \wedge ( \neg n)) \vee (r \wedge ( \neg m)) \neq 0
\end{array}\]
\end{theorem}

Analogously for sequences of integers $\{n[j]\}$, $\{k[j]\}$ we have the following result for $\left(\prod_{j=1}^{T} \left(\begin{array}{c}n[j]\\k[j]\end{array}\right)\right) \mod{2}$.

\begin{theorem}\label{thm:three}
\begin{eqnarray*}
\prod_{j=1}^{T} \left(\begin{array}{c}n[j]\\k[j]\end{array}\right) &&\equiv 0 \mod{2} \Leftrightarrow \\
&&(k[1] \wedge ( \neg n[1])) \vee (k[2] \wedge (\neg n[2])) \vee \dots \vee (k[T] \wedge ( \neg n[T])) \neq 0.
\end{eqnarray*}
\end{theorem}

These equivalences are simply consequences of Lucas' theorem for $p=2$ but the use of the bitwise notation will be helpful in deriving properties of binomial coefficients mod $2$. For instance, the following well-known result can easily be shown.

\begin{lemma}\label{lem:bin2n}
The central binomial coefficient $ \left(\begin{array}{c}2n\\n\end{array}\right)$ is even if and only if $n> 0$.
\end{lemma}
\begin{proof}
First note that $ \left(\begin{array}{c}0\\0\end{array}\right)=1$ is odd.
For $n>0$, let $n = 2^sr$ where $r$ is odd. Then $n \wedge \neg 2n = 2^s(r\wedge \neg 2r)$ since the $s$ least significant bits are $0$.  As $r$ is odd, $r\wedge \neg 2r \neq 0$, and the conclusion follows from 
Theorem \ref{thm:one}.
\end{proof}

As is common in formulas involving logical operators, $\neg$ has higher precedence than $\wedge$ which in turn has higher precedence than $\vee$.

\section{Run length transform}

For a sequence $\{b_i\}$ of bits, let the {\em $1$-runs} $R$ denote the sequence of lengths of consecutive $1$'s in the sequence. For example, for the bits $011011100111$, the consecutive $1$'s have lengths $2$, $3$ and $3$ and $R = (2,3,3)$.

The run length transform of sequences of numbers is defined as follows \cite{sloane:CA:2018}:
\begin{definition}
The {\em run length transform} of $\{S_n\}_{n\geq 0}$ is given by $\{T_n\}_{n\geq 0}$,
where $T_0 = S_0$ and for $n> 0$, $T_n = \Pi_{i\in R} S_i$ with $R$ being the $1$-runs of the binary representation of $n$.
\end{definition}

In the rest of this paper, as in  Ref. \cite{sloane:CA:2018}, we assume that $S_0 = 1$.
As an example, suppose $n = 463$, which is $111001111$ in binary.  It has a run of $3$ 1's and a run of $4$ 1's, and thus $T_n = S_3\cdot S_4$.  
Some fixed points of the run length transform include the sequences $\{1,0,0,\dots\}$ and $\{1,1,1,\dots \}$.
In Ref.\ \cite{sloane:CA:2018}, the following result is proved about the run length transform:

\begin{theorem} \label{thm:rlt}
Let $\{S_n\}_{n\geq 0}$ be defined by the recurrence $S_{n+1} = d_0S_n + d_1S_{n-1}$ with initial conditions $S_0 = 1$, $S_1 = c_1$. Then
the run length transform of $\{S_n\}$ is given by $\{T_n\}_{n\geq 0}$ satisfying 
$T_0 = 1$, $T_{2n} = T_n$, $T_{4n+1} = c_1T_n$, and $T_{4n+3} = d_0T_{2n+1}+d_1T_n$.
\end{theorem}

Note that the sequence $\{S_n\}$ may not uniquely define the values of $d_0$ and $d_1$ in Theorem \ref{thm:rlt}.  For instance, for the sequence $\{S_n\} = \{1,2,4,8,\dots \}$, $d_0$ and $d_1$ can be chosen to be any integers such that
$2d_0+d_1 = 4$. On the other hand, note that the run length transform is injective (one-to-one), since $S_i = T_{2^i-1}$ for $i\geq0$ and  the sequence $\{S_n\}$ can be derived from the corresponding sequence $\{T_n\}$.

\section{Recurrence relations of products of binomial coefficients mod $2$}
\begin{definition}\label{def:F}
Consider integers $a_i$, $i=1,\dots , 4$, with $0 \leq a_1+a_2$, and $0\leq a_3+a_4$. 
Define $F(n,k) = \left(\begin{array}{c}a_1n+a_2k\\a_3n+a_4k\end{array}\right) \left(\begin{array}{c}n\\k\end{array}\right) \mod 2$ and $g(n,k) = ((a_3n+a_4k) \wedge \neg (a_1n+a_2k)) \vee (k \wedge \neg n)$.
\end{definition}
By Theorem \ref{thm:two}, $F(n,k) = 1$ if and only if $g(n,k) = 0$. One direct consequence of this is a property we will often use: if $g(m,r) = wg(n,k)$ for some $w\neq 0$, then $F(m,r) = F(n,k)$. The functions $F$ and $g$ depend on the integers $a_i$ whose values are clear from context. We next show that $F$ satisfies various recurrence relations.
In the formulas below, the arithmetical operations $+$ and $\times$ have precedence over the bitwise logical operators $\wedge$, $\vee$ and $\neg$.

\begin{theorem}\label{thm:fandg}
The following relations hold for the function $F$:
\begin{itemize}
\item $F(n,k) = 0$ if $k > n$,
\item $F(2^rn,2^rk) = F(n,k)$ for $r > 0$,
\item $F(2n,2k+1) = F(4n+1,4k+2) = F(4n+1,4k+3) = F(4n+2,4k+1) = F(4n+2,4k+3) =F(4n,4k+1) = F(4n,4k+2) = F(4n,4k+3) = 0$,
\item Suppose $a_3\in\{0,1\}$. If $a_1 = 1$ or $a_3 = 0$, then  $F(4n+1,4k)= F(4n+3,4k) = F(2n+1,2k) = F(n,k)$,
\item If  $a_3 \wedge \neg a_1 \equiv 0 \mod 4$ and $0\leq a_1, a_3 < 4$, then  $F(4n+1,4k) = F(n,k)$,
\item If  $a_3 \wedge \neg a_1 \not\equiv 0 \mod 4$, then  $F(4n+1,4k) = 0$,
\item If $3a_3 \wedge \neg 3a_1 \not\equiv 0 \mod 4$, then  $F(4n+3,4k) = 0$,
\item If  $a_3 \wedge \neg a_1 \not\equiv 0 \mod 2$, then  $F(2n+1,2k) = 0$.

\end{itemize}
\end{theorem}

\begin{proof}
If $k > n$, then by definition  $\left(\begin{array}{c}n\\k\end{array}\right) = 0$  and thus $F(n,k) = 0$.

Note that $g(2n,2k) = (2(a_3n+a_4k) \wedge \neg(2(a_1n+a_2k))) \vee (2k \wedge \neg 2n) = 
2g(n,k)$ since the least significant bit is $0$, i.e. $F(2n,2k) = F(n,k)$. 

Next $F(n,k) = 0$ if  $\left(\begin{array}{c}n\\k\end{array}\right) \equiv 0\mod 2$, i.e. if $(k \wedge \neg n) > 0$.  It is easy to see that
 $F(2n,2k+1) = F(4n+1,4k+2) = F(4n+1,4k+3) = F(4n+2,4k+1) = F(4n+2,4k+3) = 0$ and $F(4n,4k+i) = 0$ for $ 1\leq i \leq 3$.

Since $g(4n+1,4k) = (4(a_3n+a_4k)+a_3 \wedge \neg 4(a_1n+a_2k)+a_1) \vee (4k \wedge \neg (4n+1))$ and $4k \wedge \neg (4n+1)\equiv 0 \mod 4$, the least significant 2 bits of $g(4n+1,4k)$ are equal to $a_3 \wedge \neg a_1 \mod 4$.
This means that if $a_1=1$ or $a_3 = 0$, then $g(4n+1,4k) = 4g(n,k)$ and $F(4n+1,4k) = F(n,k)$. 
If $a_3 \wedge \neg a_1 \not\equiv 0 \mod 4$, then $F(4n+1,4k) = 0$. 

Similarly, we write
$g(4n+3,4k) = (4(a_3n+a_4k)+3a_3 \wedge \neg(4(a_1n+a_2k)+3a_1)) \vee (4k \wedge \neg (4n+3)) $ which is equal to $4g(n,k)$ if $a_1 = 1$ or  $a_3 = 0$, i.e. $F(4n+3,4k) = F(n,k)$ if $a_1 = 1$ or  $a_3 = 0$ and $F(4n+3,4k) = 0$ if $3a_3 \wedge \neg 3a_1 \not\equiv 0 \mod 4$.

Finally,
$g(2n+1,2k) = (2(a_3n+a_4k)+a_3 \wedge \neg(2(a_1n+a_2k)+a_1)) \vee (2k \wedge \neg (2n+1)) $. Thus $F(2n+1, 2k) = 0$  if $a_3 \wedge \neg a_1 \not\equiv 0 \mod 2$.  If $a_1=1$ or $a_3=0$, then $g(2n+1,2k) = 2g(n,k)$ and $F(2n+1,2k) = F(n,k)$.
\end{proof}

\section{Sums of products of binomial coefficients mod $2$}

In this section, we show that for various values of $a_i$'s,  the sequence $a(n) = \sum_{k=0}^n F(n,k)$ corresponds to the run length transforms of well-known sequences\footnote{Again, $a(n)$ depends on the integers $a_i$ whose values are deduced from context.}.

In particular, we show that the sequences
$$\sum_{k=0}^{n} \left[\left(\begin{array}{c}n-k\\2k\end{array}\right) \left(\begin{array}{c}n\\k\end{array}\right) \mod 2\right],$$
$$\sum_{k=0}^{n} \left[\left(\begin{array}{c}n+k\\n-k\end{array}\right) \left(\begin{array}{c}n\\k\end{array}\right) \mod 2\right],$$
$$\sum_{k=0}^n \left[\left(\begin{array}{c}n+2k\\2n-k\end{array}\right) \left(\begin{array}{c}n\\k\end{array}\right) \mod 2\right],$$ and 
$$\sum_{k=0}^n \left[\left(\begin{array}{c}n-k\\6k\end{array}\right) \left(\begin{array}{c}n\\k\end{array}\right) \mod 2\right]$$ are the run length transform of the Fibonacci numbers, the positive integers, the extended Lucas numbers and Narayana's cows sequence, respectively.

It is clear that $a(n)$ is upper bounded by $a(n)\leq \sum_{k=0}^n \left[\left(\begin{array}{c}n\\k\end{array}\right)\mod 2\right]$  with equality when $a_1=a_4 = 1$, $a_2 = a_3 = 0$ or when $a_1=a_2=a_3=a_4=1$.
The sequence $\alpha(n) =  \sum_{k=0}^n \left[ \left(\begin{array}{c}n\\k\end{array}\right)\mod 2\right]$ is known as Gould's sequence or Dress' sequence and is the run length transform of the positive powers of $2$: $ 1,2,4,8,16,32,\dots$ (see  OEIS \cite{oeis} sequence \OEIS{A001316}).

\begin{lemma} \label{lem:a2n}
The sequence $a(n)$ satisfies the following properties:
\begin{itemize}
\item
$a(0) = 1$,
\item $a(2^rn) = a(n)$ for $r > 0$,
\item Suppose $a_3\in\{0,1\}$. If $a_1=1$ or $a_3=0$, then $a(4n+1) = a(n) + \sum_{k=0}^{n} F(4n+1,4k+1)$, 
$a(4n+3) = a(n) +  \sum_{m=1}^{3} \sum_{k=0}^{n} F(4n+3,4k+m)$, 
and $a(2n+1) = a(n) + \sum_{k=0}^{n} F(2n+1,2k+1)$, 
\item If  $a_3 \wedge \neg a_1 \not\equiv 0 \mod 4$, then $a(4n+1) =  \sum_{k=0}^{n} F(4n+1,4k+1)$,
\item If  $3a_3 \wedge \neg 3a_1 \not\equiv 0 \mod 4$, then  $a(4n+3) =  \sum_{m=1}^{3} \sum_{k=0}^{n} F(4n+3,4k+m)$,
\item If  $a_3 \wedge \neg a_1 \not\equiv 0 \mod 2$, then $a(2n+1) =  \sum_{k=0}^{n} F(2n+1,2k+1)$.

\end{itemize}
\end{lemma}
\begin{proof}
First, note that $a(0)$ is trivially equal to $1$.

Next, $a(2n) = \sum_{k=0}^{2n} F(2n,k) = \sum_{k=0}^{n} F(2n,2k) + \sum_{k=0}^{n-1} F(2n,2k+1)$ which is equal to $a(n)$ by Theorem \ref{thm:fandg}.

Suppose $a_3\in \{0,1\}$ and $a_1=1$ or $a_3=0$. By Theorem \ref{thm:fandg}, several of the terms $F(n,k)$ vanish or can be rewritten and the following relations hold. Firstly,
$a(4n+1) = \sum_{m=0}^{3} \sum_{k=0}^{n} F(4n+1,4k+m) - F(4n+1,4n+2) - F(4n+1,4n+3) = \sum_{m=0}^{3} \sum_{k=0}^{n} F(4n+1,4k+m) =\sum_{k=0}^{n} F(n,k) +\sum_{k=0}^{n} F(4n+1,4k+1)$.

Secondly,
$a(4n+3) = \sum_{m=0}^{3} \sum_{k=0}^{n} F(4n+3,4k+m)$ which is equal to $\sum_{k=0}^{n} F(n,k) +  \sum_{m=1}^{3} \sum_{k=0}^{n} F(4n+3,4k+m) $.

Next,
$a(2n+1) = \sum_{k=0}^{2n+1} F(2n+1,k) = \sum_{k=0}^{n} F(2n+1,2k) + \sum_{k=0}^{n} F(2n+1,2k+1) = \sum_{k=0}^{n} F(n,k)  + \sum_{k=0}^{n} F(2n+1,2k+1)$. The other cases follows similarly from 
Theorem \ref{thm:fandg}.
\end{proof}

In particular, if  $a_1=1$ or $a_3=0$ then  $\sum_{k=0}^{n} F(2n+1,2k+1) = a(2n+1)-a(n)$, an equation that we will often use in what follows.

\subsection{Run length transform of the Fibonacci sequence}
First, consider the case $a_1 = 1$, $a_2 = -1$, $a_3 = 0$, $a_4 = 2$.

\begin{lemma} \label{lem:fandgfib}
For $a_1 = 1$, $a_2 = -1$, $a_3 = 0$, $a_4 = 2$, the following relations hold for the function $F$:
\begin{itemize}
\item $F(4n+1,4k+1) = F(4n+3,4k+3) =  0$,
\item $F(4n+3,4k+1) = F(n,k)$,
\item $F(4n+3,4k+2) = F(2n+1,2k+1)$.
\end{itemize}
\end{lemma}
\begin{proof}
First, $g(4n+1,4k+1) = (8k+2 \wedge \neg(4(n-k))) \vee (4k+1 \wedge \neg 4n+1)  \neq 0$, i.e
$F(4n+1,4k+1) = 0$. 

Next, $g(4n+3,4k+1) = (8k+2 \wedge \neg(4(n-k)+2)) \vee (4k+1 \wedge \neg 4n+3) = (4(2k\vee n-k)) \vee 4(k\wedge \neg n) = 4g(n,k)$, i.e. $F(4n+3,4k+1) = F(n,k)$.

Note that $(4k+2 \wedge \neg 4n+3) = (4k \wedge \neg 4n) = 2(2k \wedge \neg 2n)$ and $(2k+1 \wedge \neg 2n+1) = (2k \wedge \neg 2n)$.
Similarly $(4k+3 \wedge \neg 4n+3) = (4k \wedge \neg 4n) = 2(2k+1 \wedge \neg 2n+1)$.

Furthermore, $g(4n+3,4k+2) = (8k+4 \wedge \neg(4(n-k)+1)) \vee (4k+2 \wedge \neg 4n+3) = 2[(4k+2 \wedge \neg 2(n-k)) \vee (2k+1 \wedge \neg 2n+1)] = 2g(2n+1,2k+1)$
where we have used the fact that  $(8k+4 \wedge \neg(4(n-k)+1)) =  (8k+4 \wedge \neg(4(n-k)))$. This implies that $F(4n+3,4k+2) = F(2n+1,2k+1)$. 

Finally, $g(4n+3,4k+3) = (8k+6 \wedge \neg(4(n-k))) \vee (4k+3 \wedge \neg 4n+3)$. Since $8k+6 \wedge \neg 4(n-k) \neq 0$, this implies that $F(4n+3,4k+3) = 0$.
\end{proof}

\begin{theorem} \label{thm:fibonacci}
Let $a(n) =  \sum_{k=0}^n \left[\left(\begin{array}{c}n-k\\2k\end{array}\right) \left(\begin{array}{c}n\\k\end{array}\right) \mod 2\right]$. Then
$a(n)$ satisfies the equations $a(0) =1$, $a(2n) = a(n)$, $a(4n+1) = a(n)$ and $a(4n+3) = a(2n+1) + a(n)$.
In particular, $a(n)$ is the run length transform of the Fibonacci sequence $1,1,2,3,5,8,13,\dots$.
\end{theorem}

\begin{proof}
By Lemma \ref{lem:a2n}, $a(0) =1$ and $a(2n) = a(n)$.
Next, by Lemma \ref{lem:a2n} and Lemma \ref{lem:fandgfib}, we obtain
$a(4n+1) = a(n)$.
Similarly $a(4n+3) = a(n) + \sum_{m=1}^{3} \sum_{k=0}^{n} F(4n+3,4k+m) = a(n) + \sum_{k=0}^{n} F(n,k) + F(2n+1,2k+1) = a(2n+1)+a(n)$.
By Theorem \ref{thm:rlt}, $a(n)$ is the run length transform of the Fibonacci sequence 1,1,2,3,5,8,13,...
\end{proof}
Note that in this case $a(n)$ corresponds to OEIS sequence \OEIS{A246028}.
Other values of $a_i$ can also generate the same sequence. For instance, it can be shown that the values of   $\sum_{k=0}^n \left[\left(\begin{array}{c}2k\\n-k\end{array}\right) \left(\begin{array}{c}n\\k\end{array}\right) \mod 2\right]$ , $\sum_{k=0}^n \left[\left(\begin{array}{c}n+3k\\2k\end{array}\right) \left(\begin{array}{c}n\\k\end{array}\right)\mod 2\right]$, and of 
$$\sum_{k=0}^n \left[\left(\begin{array}{c}n+3k\\n+k\end{array}\right) \left(\begin{array}{c}n\\k\end{array}\right) \mod 2\right]$$ all correspond to the run length transform of the Fibonacci sequence as well.

\subsection{Run length transform of the truncated Fibonacci sequence}

Next, consider the case $a_1 = a_3 = 0$, $a_2 =  3$, $a_4 = 1$.

\begin{lemma} \label{lem:fandgfibt}
For $a_1 = a_3 = 0$, $a_2 =  3$, $a_4 = 1$, the following relations hold for the function $F$:
\begin{itemize}
\item $F(4n+1,4k+1) =  F(4n+3,4k+1)  =  F(n,k)$,
\item $F(4n+3,4k+2) = F(2n+1,2k+1)$,
\item $F(4n+3,4k+3) = 0$.
\end{itemize}
\end{lemma}
\begin{proof}
$g(4n+1,4k+1) = (4k+1 \wedge \neg(12k+3)) \vee (4k+1 \wedge \neg 4n+1) =   (4k \wedge \neg 12k) \vee (4k \wedge \neg 4n) $, i.e
$F(4n+1,4k+1) = F(n,k)$. 

$g(4n+3,4k+1) = (4k+1 \wedge \neg(12k+3)) \vee (4k+1 \wedge \neg 4n+3) = (4k \wedge \neg 12k) \vee (4k \wedge \neg 4n) $ and $F(4n+3,4k+1) = F(n,k)$.

Note that $(4k+2 \wedge \neg 4n+3) = (4k \wedge \neg 4n) = 2(2k \wedge \neg 2n)$ and $(2k+1 \wedge \neg 2n+1) = (2k \wedge \neg 2n)$.
Similarly $(4k+3 \wedge \neg 4n+3) = (4k \wedge \neg 4n) = 2(2k+1 \wedge \neg 2n+1)$.

$g(4n+3,4k+2) = (4k+2 \wedge \neg(12k+6)) \vee (4k+2 \wedge \neg 4n+3) = 2[(2k+1 \wedge \neg 6k+3) \vee (2k+1 \wedge \neg 2n+1)]$. This implies that $F(4n+3,4k+2) = F(2n+1,2k+1)$. 

$g(4n+3,4k+3) = (4k+3 \wedge \neg(12k+6)) \vee (4k+3 \wedge \neg 4n+3)$. Since $4k+3 \wedge \neg 12k+6 \neq 0$
This implies that $F(4n+3,4k+3) = 0$
\end{proof}

\begin{theorem} \label{thm:fibonacci_truncated}
Let $a(n) =  \sum_{k=0}^n \left[\left(\begin{array}{c}3k\\k\end{array}\right) \left(\begin{array}{c}n\\k\end{array}\right) \mod 2\right]$. Then
$a(n)$ satisfies the equations $a(0) = 1$, $a(2n) = a(n)$, $a(4n+1) = 2a(n)$ and $a(4n+3) = a(2n+1) + a(n)$.
In particular, $a(n)$ is the run length transform of the truncated Fibonacci sequence $1,2,3,5,8,13,\dots$.
\end{theorem}

\begin{proof}
By Lemma \ref{lem:a2n}, $a(0) = 1$ and $a(2n) = a(n)$.
Next, 
 by Lemma \ref{lem:a2n} and Lemma \ref{lem:fandgfibt},
$a(4k+1)= 2a(n)$.
Similarly, $a(4n+3) = a(n) + \sum_{m=1}^{3} \sum_{k=0}^{n} F(4n+3,4k+m) =a(n) + \sum_{k=0}^{n} F(n,k) + F(2n+1,2k+1)=  a(2n+1)+a(n)$.
By Theorem \ref{thm:rlt}, $a(n)$ is the run length transform of the truncated Fibonacci sequence $1,2,3,5,8,13,\dots$.
\end{proof}
Note that in this case $a(n)$ corresponds to OEIS sequence \OEIS{A245564}. This sequence is also equal to 
 $\sum_{k=0}^n \left[\left(\begin{array}{c}3k2^m\\k2^m\end{array}\right) \left(\begin{array}{c}n\\k\end{array}\right) \mod 2\right]$ and
 $$\sum_{k=0}^n \left[\left(\begin{array}{c}3k2^m\\2k2^m\end{array}\right) \left(\begin{array}{c}n\\k\end{array}\right) \mod 2\right]$$ for all integers $m\geq 0$.

\subsection{Run length transform of $\{1,1,2,4,8,16,32,\dots\}$}
Consider the case $a_1 = 1$, $a_2 = a_3 = 0$, $a_4 = 2$.

\begin{lemma} \label{lem:fandgpow2}
For  $a_1 = 1$, $a_2 = a_3 = 0$, $a_4 = 2$, the following relations hold for the function $F$:
\begin{itemize}
\item $F(4n+1,4k+1) = 0$,
\item $F(4n+3,4k+1)  =  F(n,k)$,
\item $F(4n+3,4k+2) = F(4n+3,4k+3)  = F(2n+1,2k+1)$.
\end{itemize}
\end{lemma}
\begin{proof}
$g(4n+1,4k+1) = (8k+2 \wedge \neg(4n+1)) \vee (4k+1 \wedge \neg 4n+1) \neq 0$, i.e
$F(4n+1,4k+1) = 0$. 

$g(4n+3,4k+1) = (8k+2 \wedge \neg(4n+3)) \vee (4k+1 \wedge \neg 4n+3) = 4[(2k \wedge \neg n) \vee (k \wedge \neg n)]$ where we use the fact that
$ (8k+2 \wedge \neg(4n+3))=(8k \wedge \neg(4n))$
and thus $F(4n+3,4n+1) = F(n,k)$.

Note that $(4k+2 \wedge \neg 4n+3) = (4k \wedge \neg 4n) = 2(2k \wedge \neg 2n)$ and $(2k+1 \wedge \neg 2n+1) = (2k \wedge \neg 2n)$.
Similarly $(4k+3 \wedge \neg 4n+3) = (4k \wedge \neg 4n) = 2(2k+1 \wedge \neg 2n+1)$.

$g(4n+3,4k+2) = (8k+4 \wedge \neg(4n+3)) \vee (4k+2 \wedge \neg 4n+3) = 2[(4k+2 \wedge \neg 2n+1) \vee (2k+1 \wedge \neg 2n+1)]$, where we use the fact that
$ (8k+4 \wedge \neg(4n+3))  =  (8k+4 \wedge \neg(4n+2)) $. 
This implies that $F(4n+3,4k+2) = F(2n+1,2k+1)$. 

$g(4n+3,4k+3) = (8k+6 \wedge \neg(4n+3)) \vee (4k+3 \wedge \neg 4n+3) = 2[ (4k+2 \wedge \neg(2n+1)) \vee (2k+1 \wedge \neg 2n+1)]$.This implies that $F(4n+3,4k+3) = F(2n+1,2k+1)$. 

\end{proof}

\begin{theorem} \label{thm:recursivepow2}
Let $a(n) =  \sum_{k=0}^n \left[\left(\begin{array}{c}n\\2k\end{array}\right) \left(\begin{array}{c}n\\k\end{array}\right) \mod 2\right]$. Then
$a(n)$ satisfies the equations $a(0) =1$, $a(2n) = a(n)$, $a(4n+1) = a(n)$ and $a(4n+3) = 2a(2n+1)$.
In particular, $a(n)= \{1, 1, 1, 2, 1, 1, 2, 4, 1, 1, 1, 2, \dots \}$ is the run length transform of the sequence $1,1,2,4,8,16,32,\dots$, i.e. $1$ plus the positive powers of $2$.
\end{theorem}

\begin{proof}
By Lemma \ref{lem:a2n}, $a(0) = 1$ and $a(2n) = a(n)$.
Next, by Lemma \ref{lem:a2n} and Lemma \ref{lem:fandgpow2},
$a(4n+1) =  a(n)$.
Similarly,
$a(4n+3) =  a(n) + \sum_{m=1}^{3} \sum_{k=0}^{n} F(4n+3,4k+m) =  2a(n) + 2\sum_{k=0}^{n} F(2n+1,2k+1) = 2a(2n+1)$.
By Theorem \ref{thm:rlt}, $a(n) $ is the run length transform of the sequence $1,1,2,4,8,16,32,\dots$.
\end{proof}

\subsection{Run length transform of $\{1,2,2,2,2,2,\dots\}$}
Consider the case $a_1 = 1$, $a_2 = a_4 = 2$, $a_3 = 0$.

\begin{lemma} \label{lem:fandg1222}
For $a_1 = 1$, $a_2 = a_4 = 2$, $a_3 = 0$, the following relations hold for the function $F$:
\begin{itemize}
\item $F(4n+1,4k+1) = F(n,k)$,
\item $F(4n+3,4k+1)  =  F(4n+3,4k+3) = 0$,
\item $F(4n+3,4k+2) = F(2n+1,2k+1)$.
\end{itemize}
\end{lemma}
\begin{proof}
$g(4n+1,4k+1) = (8k+2 \wedge \neg(4n+8k+3)) \vee (4k+1 \wedge \neg 4n+1) = 4[(2k \wedge \neg(n+2k)) \vee (k \wedge \neg n)]$, i.e
$F(4n+1,4k+1) = F(n,k)$. 

$g(4n+3,4k+1) = (8k+2 \wedge \neg(4n+8k+5)) \vee (4k+1 \wedge \neg 4n+3) \neq 0$, i.e. $F(4n+3,4k+1)= 0$.

Note that $(4k+2 \wedge \neg 4n+3) = (4k \wedge \neg 4n) = 2(2k \wedge \neg 2n)$ and $(2k+1 \wedge \neg 2n+1) = (2k \wedge \neg 2n)$.
Similarly $(4k+3 \wedge \neg 4n+3) = (4k \wedge \neg 4n) = 2(2k+1 \wedge \neg 2n+1)$.

$g(4n+3,4k+2) = (8k+4 \wedge \neg(4n+8k+7)) \vee (4k+2 \wedge \neg 4n+3) = 2[(4k+2 \wedge \neg 2n+4k+3) \vee (2k+1 \wedge \neg 2n+1)]$
where we use $(8k+4 \wedge \neg(4n+8k+7)) = (8k+4 \wedge \neg(4n+8k+6))$.
This implies that $F(4n+3,4k+2) = F(2n+1,2k+1)$. 

$g(4n+3,4k+3) = (8k+6 \wedge \neg(4n+8k+9)) \vee (4k+3 \wedge \neg 4n+3) \neq 0$, i.e. $F(4n+3,4k+3) = 0$. 

\end{proof}

\begin{theorem} \label{thm:recursive1222}
Let $a(n) =  \sum_{k=0}^n \left[\left(\begin{array}{c}n+2k\\2k\end{array}\right) \left(\begin{array}{c}n\\k\end{array}\right) \mod 2\right]$. Then
$a(n)$ satisfies the equations $a(0) =1$, $a(2n) = a(n)$, $a(4n+1) = 2a(n)$ and $a(4n+3) = a(2n+1)$.
In particular, $a(n) = \{1, 2, 2, 2, 2, 4, 2, 2, 2, 4, \dots \}$ is the run length transform of the sequence $1,2,2,2,2,2,2,\dots$ (OEIS \OEIS{A040000}).
\end{theorem}

\begin{proof}
By Lemma \ref{lem:a2n}, $a(0) = 1$, $a(2n) = a(n)$.
Next, 
by Lemma \ref{lem:a2n} and Lemma \ref{lem:fandg1222},
$a(4n+1) = a(n) + \sum_{k=0}^{n} F(n,k) = 2a(n)$.
Similarly, $a(4n+3) =  a(n) + \sum_{m=1}^{3} \sum_{k=0}^{n} F(4n+3,4k+m) = a(n) + \sum_{k=0}^{n}  F(2n+1,2k+1) = a(2n+1)$.
By Theorem \ref{thm:rlt}, $a(n)$ is the run length transform of the Fibonacci sequence $1,2,2,2,2,\dots$.
\end{proof}
This sequence is also generated by $  \sum_{k=0}^n \left[ \left(\begin{array}{c}n+2k\\n\end{array}\right) \left(\begin{array}{c}n\\k\end{array}\right) \mod 2\right]$.

\subsection{Run length transform of the positive integers}
OEIS sequence \OEIS{A106737} is defined as $a(n) = \sum_{k=0}^{n} \left[ \left(\begin{array}{c}n+k\\n-k\end{array}\right) \left(\begin{array}{c}n\\k\end{array}\right) \mod 2\right]$.  It was noted that the following recursive relationships appear to hold:
$a(2n) = a(n)$, $a(4n+1) = 2a(n)$ and $a(4n+3) = 2a(2n+1) - a(n)$.  In this section we show that this is indeed the case.

Let $a_1, a_2, a_3 = 1$ and $a_4 = -1$, i.e. $F(n,k) = \left(\begin{array}{c}n+k\\n-k\end{array}\right) \left(\begin{array}{c}n\\k\end{array}\right) \mod 2$ and $g(n,k) = ((n-k) \wedge \neg (n+k)) \vee (k \wedge \neg n)$.

\begin{lemma}\label{lem:fandg2}
For  $a_1, a_2, a_3 = 1$ and $a_4 = -1$, the following relations hold for the function $F$:
\begin{itemize}
\item $F(4n+1,4k+1) = F(n,k)$,
\item $F(4n+3,4k+1) = 0$,
\item $F(4n+3,4k+2) = F(4n+3,4k+3) = F(2n+1,2k+1)$.
\end{itemize}
\end{lemma}
\begin{proof}

$g(4n+1,4k+1) = (4(n-k) \wedge \neg(4(n+k)+2)) \vee (4k+1 \wedge \neg 4n+1) =  4((n-k) \wedge \neg(n+k)) \vee 4(k \wedge \neg n) $, i.e.
$F(4n+1,4k+1) = F(n,k)$.

$g(4n+3,4k+1) = (4(n-k)+2 \wedge \neg(4(n+k+1))) \vee (4k+1 \wedge \neg 4n+3) > 0 $ since $(4(n-k)+2 \wedge \neg(4(n+k+1))) > 0$, i.e. $F(4n+3,4k+1) = 0$.

Note that $(4k+2 \wedge \neg 4n+3) = (4k \wedge \neg 4n) = 2(2k \wedge \neg 2n)$ and $(2k+1 \wedge \neg 2n+1) = (2k \wedge \neg 2n)$.
Similarly $(4k+3 \wedge \neg 4n+3) = (4k \wedge \neg 4n) = 2(2k+1 \wedge \neg 2n+1)$.

$g(4n+3,4k+2) = (4(n-k)+1 \wedge \neg(4(n+k+1)+1)) \vee (4k+2 \wedge \neg 4n+3) = 2[(2(n-k) \wedge \neg 2(n+k+1)) \vee (2k+1 \wedge \neg 2n+1)]$ where we
use $4(n-k)+1 \wedge \neg(4(n+k+1)+1) = 4(n-k) \wedge \neg(4(n+k+1))$. 
This implies that $F(4n+3,4k+2) = F(2n+1,2k+1)$. 

$g(4n+3,4k+3) = (4(n-k) \wedge \neg(4(n+k+1)+2)) \vee (4k+3 \wedge \neg 4n+3) = 2[(2(n-k) \wedge \neg 2(n+k+1)) \vee (2k+1 \wedge \neg 2n+1)]$ where we use
$4(n-k) \wedge \neg(4(n+k+1)+2) = 4(n-k) \wedge \neg(4(n+k+1))$, and thus
$F(4n+3,4k+3) = F(2n+1,2k+1)$.
\end{proof}

\begin{theorem} \label{thm:recursive2}
For OEIS sequence \OEIS{A106737},
$a(0) = 1$, $a(2n) = a(n)$, $a(4n+1) = 2a(n)$ and $a(4n+3) = 2a(2n+1) - a(n)$. Furthermore, $a(n)$ is the run length transform of the positive integers.
\end{theorem}
\begin{proof}
As before, by Lemma \ref{lem:a2n}, $a(0) = 1$ and $a(2n) = a(n)$. Next 
by Lemma \ref{lem:a2n} and Lemma \ref{lem:fandg2},  $a(4n+1) = a(n) + \sum_{k=0}^{n} F(n,k) = 2a(n)$.
Similarly, $a(4n+3) = a(n) +  \sum_{k=0}^{n}  F(2n+1,2k+1) + F(2n+1,2k+1) = 2a(2n+1)-a(n)$.
By Theorem \ref{thm:rlt}, $a(n)$ is the run length transform of the positive integers $1,2,3,4, \dots$
\end{proof}
This sequence is also generated by each of the following expressions: 

$\sum_{k=0}^{n} \left[\left(\begin{array}{c}n+k\\2k\end{array}\right) \left(\begin{array}{c}n\\k\end{array}\right) \mod 2\right]$, $\sum_{k=0}^{n} \left[\left(\begin{array}{c}n+2k\\k\end{array}\right) \left(\begin{array}{c}n\\k\end{array}\right) \mod 2\right]$ and 

$\sum_{k=0}^{n} \left[ \left(\begin{array}{c}n+2k\\n+k\end{array}\right) \left(\begin{array}{c}n\\k\end{array}\right) \mod 2\right]$.

\subsection{A fixed point of the run length transform}
The all ones sequence $1,1,1,\dots$ (OEIS sequence \OEIS{A000012}) is a fixed point of the run length transform. We next show that it is also expressible as sums of products of binomial coefficients mod 2.
To prove this, we look at the case $a_1 = a_4 = 1$, $a_2 = -1$,  $a_3 = 0$.

\begin{lemma} \label{lem:fandg1111}
For $a_1 = a_4 = 1$, $a_2 = -1$,  $a_3 = 0$, the following relations hold for the function $F$:
$$F(4n+1,4k+1) = F(4n+3,4k+1) = F(4n+3,4k+2)  =  F(4n+3,4k+3) = 0.$$
\end{lemma}
\begin{proof}
$g(4n+1,4k+1) = (4k+1 \wedge \neg(4(n-k))) \vee (4k+1 \wedge \neg 4n+1) \neq 0$, i.e.
$F(4n+1,4k+1) = 0$. 

$g(4n+3,4k+1) = (4k+1 \wedge \neg(4(n-k)+2)) \vee (4k+1 \wedge \neg 4n+3) \neq 0$, i.e. $F(4n+3,4k+1)= 0$.

Note that $(4k+2 \wedge \neg 4n+3) = (4k \wedge \neg 4n) = 2(2k \wedge \neg 2n)$ and $(2k+1 \wedge \neg 2n+1) = (2k \wedge \neg 2n)$.
Similarly $(4k+3 \wedge \neg 4n+3) = (4k \wedge \neg 4n) = 2(2k+1 \wedge \neg 2n+1)$.

$g(4n+3,4k+2) = (4k+2 \wedge \neg(4(n-k)+1)) \vee (4k+2 \wedge \neg 4n+3) \neq 0$, i.e. $F(4n+3,4k+2)= 0$.
$g(4n+3,4k+3) = (4k+3 \wedge \neg(4(n-k))) \vee (4k+3 \wedge \neg 4n+3) \neq 0$, i.e. $F(4n+3,4k+3) = 0$. 
\end{proof}

\begin{theorem} \label{thm:odd}
For $n, k\geq 0$, 
$\left(\begin{array}{c}n-k\\k\end{array}\right) \left(\begin{array}{c}n\\k\end{array}\right)$ is odd if and only if $k = 0$, i.e.
$$\sum_{k=0}^n \left[\left(\begin{array}{c}n-k\\k\end{array}\right) \left(\begin{array}{c}n\\k\end{array}\right) \mod 2 \right] = 1$$ for all $n$.
\end{theorem}

\begin{proof}
Define $a(n) = \sum_{k=0}^n \left[\left(\begin{array}{c}n-k\\k\end{array}\right) \left(\begin{array}{c}n\\k\end{array}\right) \mod 2\right]$.
By Lemma \ref{lem:a2n} and Lemma \ref{lem:fandg1111}, $a(0) = 1$, $a(n) = a(2n)$ and 
$a(4n+1) = a(n)$, $a(4n+3) = a(n)$.
By Theorem \ref{thm:rlt}, $a(n)$ is the run length transform of the sequence $1,1,1,1,\dots$, i.e. $a(n) = 1$ for all $n\geq 0$.
The conclusion then follows since $\left(\begin{array}{c}n-k\\k\end{array}\right) \left(\begin{array}{c}n\\k\end{array}\right) = 1$ when $k = 0$.
\end{proof}

Theorem \ref{thm:odd} can also be interpreted via Sierpi\'{n}ski's triangle generated by Pascal's triangle mod 2 and by looking at it as follows: if starting from the left edge of the triangle and moving $k$ steps to the right reaches a point of Sierpi\'{n}ski's triangle, then continuing moving diagonally $k$ steps must necessarily reach a void of Sierpi\'{n}ski's triangle.

\section{Third order recurrences}

\begin{definition} \label{def:mu}
Let $n$ be an odd positive integer not of the form $2^k-1$. The {\em splitting function} $\mu(n) = (a,b,m)$ returns positive integers
$a$, $b$, $m$ such that $a2^m+b = n$ with $2^m > 2b$ and $a$ is the smallest such number satisfying this.
\end{definition}

A way to describe the numbers $a$ and $b$ in Definition \ref{def:mu} is that they are obtained by splitting the binary expansion of $n$ along the first occurrence of $0$.  For instance since $413 = 110011101_2$, which can be split into `11' and `011101' which is $3$ and $29$, and thus $\mu(413) = (3, 29, 7)$. Note that $n>3b$.

Theorem \ref{thm:rlt} shows that if a sequence satisfies a second order recurrence, then we can easily determine the recurrence relations that the run length transform satisfies. This result can be generalized to $n$-th order recurrences as follows:

\begin{theorem} \label{thm:rlt-general}
Let $\{S_n\}_{n\geq 0}$ be defined by the $(k+1)$-th order recurrence $S_{n+1} = \sum_{i=0}^{k} d_iS_{n-i}$ with initial conditions $S_i = c_i$ for $i = 0,1,\dots , {k}$. Then the run length transform of $\{S_n\}$ is given by $\{T_n\}_{n\geq 0}$ satisfying 
\begin{itemize}
\item $w = 2^{k+1}$,
\item
$T_0 = c_0$,
\item $T_{2n} = T_n$,
\item $T_{wn+i} = T_i T_n$, for $i = 1,3,5,\dots , 2^k-1$,
\item $T_{wn+2^k+i} = T_{b_i}T_{\frac{wn}{2^{m_i}}+a_i}$ for $i = 1,3,5,\dots , 2^k-3$ where $\mu(2^k+i) = (a_i, b_i, m_i)$,
\item $T_{wn+w-1} = \sum_{i=0}^{k} d_i T_{2^{k-i}n+2^{k-i}-1}$.
\end{itemize}
\end{theorem}

\begin{proof}
The proof is similar to the proof of Theorem \ref{thm:rlt} (see Ref. \cite{sloane:CA:2018} for a proof of Theorem \ref{thm:rlt}) and is omitted.
\end{proof}

Even though the right hand side in some of the recurrence relations above is a product of $2$ terms of $\{T_n\}$, one of the terms is determined solely by the initial conditions $c_i$'s. More specifically, note that for $i = 1,3,5,\dots , 2^k-1$, $T_i$ is a product of $S_j$'s where $j\leq k$ and thus is a product of some $c_j$'s.
Similarly, for $i\leq 2^k-3$, we have $2^{k+1}-3>2^k+i>3b_i$ and therefore $b_i < 2^k$.
Thus $T_{b_i}$ is also the product of some $c_j$'s.
In particular, Theorem \ref{thm:rlt-general} for the case $k=1$ corresponds to Theorem \ref{thm:rlt}.  For $k=2$, we have the following result on third order recurrences:

\begin{corollary}\label{cor:rlt-k2}
Let $\{S_n\}_{n\geq 0}$ be defined by the recurrence $S_{n+1} = d_0S_n + d_1S_{n-1} + d_2S_{n-2}$ with initial conditions $S_0 = c_0$, $S_1 = c_1$, $S_2 = c_2$. Then
the run length transform of $\{S_n\}$ is given by $\{T_n\}_{n\geq 0}$ satisfying 
\begin{itemize}
\item $T_0 = c_0$,
\item $T_{2n} = T_n$, 
\item $T_{8n+1} = c_1T_n$,  
\item $T_{8n+3} = c_2T_n$,
\item $T_{8n+5} = c_1T_{2n+1}$, and
\item $T_{8n+7} = d_0T_{4n+3}+d_1T_{2n+1}+d_2T_n$.
\end{itemize}
\end{corollary}

\begin{theorem}
The following relations hold for the function $F$ as defined in Definition \ref{def:F}:
\begin{itemize}
\item $F(8n,8k+i) = 0$ for $i=1,\dots , 7$,
\item $F(8n+1,8k+i) = 0$ for $i =2,\dots ,7$,
\item $F(8n+3,8k+i) = 0$ for $i = 4,5,6,7$,
\item $F(8n+5,8k+i) = 0$ for $i = 2,3,6,7$,
\item If  $a_3 \wedge \neg a_1 \equiv 0 \mod 8$ and $0\leq a_1, a_3 < 8$, then  $F(8n+1,8k) = F(n,k)$,
\item If  $a_3 \wedge \neg a_1 \not\equiv 0 \mod 8$, then  $F(8n+1,8k) = 0$,
\item If $3a_3 \wedge \neg 3a_1 \equiv 0 \mod 8$ and $0\leq 3a_1, 3a_3 < 8$, then  $F(8n+3,8k) = F(n,k)$,
\item If $3a_3 \wedge \neg 3a_1 \not\equiv 0 \mod 8$, then  $F(8n+3,8k) = 0$,
\item If $5a_3 \wedge \neg 5a_1 \equiv 0 \mod 8$ and $0\leq 5a_1, 5a_3 < 8$, then  $F(8n+5,8k) = F(n,k)$,
\item If $5a_3 \wedge \neg 5a_1 \not\equiv 0 \mod 8$, then  $F(8n+5,8k) = 0$.
\end{itemize}
\end{theorem}
\begin{proof}
The proof is similar to the proof of Theorem \ref{thm:fandg}.
\end{proof}

\begin{lemma} \label{lem:a2n_order3}
The sequence $a(n)$ satisfies the following properties:
\begin{itemize}
\item If  $a_3 \wedge \neg a_1 \equiv 0 \mod 8$ and $0\leq a_1, a_3 < 8$,  then $a(8n+1) = a(n) + \sum_{k=0}^{n} F(8n+1,8k+1)$, 
\item If $3a_3 \wedge \neg 3a_1 \equiv 0 \mod 8$ and $0\leq 3a_1, 3a_3 < 8$, then $a(8n+3) = a(n) +  \sum_{m=1}^{3} \sum_{k=0}^{n} F(8n+3,8k+m)$,  
\item  If $5a_3 \wedge \neg 5a_1 \equiv 0 \mod 8$ and $0\leq 5a_1, 5a_3 < 8$, then $a(8n+5) = a(n) +  \sum_{m\in \{1,4,5\}}\sum_{k=0}^{n} F(8n+5,8k+m)$,  
\item If  $a_3 \wedge \neg a_1 \not\equiv 0 \mod 8$, then $a(8n+1) =  \sum_{k=0}^{n} F(8n+1,8k+1)$,
\item If  $3a_3 \wedge \neg 3a_1 \not\equiv 0 \mod 8$, then  $a(8n+3) =  \sum_{m=1}^{3} \sum_{k=0}^{n} F(8n+3,8k+m)$,
\item If  $5a_3 \wedge \neg 5a_1 \not\equiv 0 \mod 8$, then  $a(8n+5) =  \sum_{m\in\{1,4,5\}} \sum_{k=0}^{n} F(8n+5,8k+m)$.

\end{itemize}
\end{lemma}
\begin{proof}
The proof is similar to the proof of Lemma \ref{lem:a2n}.
\end{proof}

\subsection{Run length transform of Narayana's cows sequence}
Narayana's cows sequence (OEIS \OEIS{A000930}) $\{b_n: n \geq 0\}$ is defined as 
$b_0 = b_1 = b_2 = 1$, $b_{n} = b_{n-1}+b_{n-3}$. The first few terms are:
$1$, $1$, $1$, $2$, $3$, $4$, $6$, $9$, $13$, $19$, $28$, $41$, $60$, $88$, $129$, $189$, $277$, $406$, $595$, $\dots$
The following results show that $$a(n) =  \sum_{k=0}^n \left[\left(\begin{array}{c}n-k\\6k\end{array}\right) \left(\begin{array}{c}n\\k\end{array}\right) \mod 2\right]$$ is the run length transform of Narayana's cows sequence.

\begin{lemma} \label{lem:fandgcows}
For  $a_1 = 1$, $a_2 = -1$, $a_3 = 0$, $a_4 = 6$, the following relations hold for the function $F$:
\begin{itemize}
\item $F(8n+1,8k+1) = 0$,
\item $F(8n+3,8k+i) = 0$ for $1\leq i \leq 3$,
\item $F(8n+5,8k+1)  = F(8n+5,8k+5) = 0$,
\item $F(8n+5,8k+4) = F(2n+1,2k+1)$,
\item $F(8n+7,8k+i) = 0$ for $i \in \{3,5,6,7\}$,
\item $F(8n+7,8k) = F(8n+7,8k+1) = F(n,k)$,
\item $F(8n+7,8k+2) = F(4n+3,4k+1)$,
\item $F(8n+7,8k+4) = F(4n+3,4k+2)$,
\item $F(4n+3,4k+3) = 0$.
\end{itemize}
\end{lemma}

\begin{proof}
We can write out the following expressions for $g$: 
\begin{itemize}
\item $g(8n+1,8k+1) = (48k+6 \wedge \neg(8(n-k))) \vee (8k+1 \wedge \neg 8n+1) \neq 0$, i.e
$F(8n+1,8k+1) = 0$. 
\item $g(8n+3,8k+1) = (48k+6 \wedge \neg(8(n-k)+2)) \vee (8k+1 \wedge \neg 8n+3) \neq 0$, i.e. $F(8n+3,8k+1) = 0$.
\item $g(8n+3,8k+2) = (48k+12 \wedge \neg(8(n-k)+1)) \vee (8k+2 \wedge \neg 8n+3) \neq 0$, i.e. $F(8n+3,8k+2) = 0$.
\item $g(8n+3,8k+3) = (48k+18 \wedge \neg(8(n-k))) \vee (8k+3 \wedge \neg 8n+3) \neq 0$, i.e. $F(8n+3,8k+3) = 0$.
\item $g(8n+5,8k+1) = (48k+6 \wedge \neg(8(n-k)+4)) \vee (8k+1 \wedge \neg 8n+5) \neq 0$, i.e. $F(8n+5,8k+1) = 0$.
\item $g(8n+5,8k+4) = (48k+24 \wedge \neg(8(n-k)+1)) \vee (8k+4 \wedge \neg 8n+5) =  (48k+24 \wedge \neg(8(n-k))) \vee (8k+4 \wedge \neg 8n+4) = 
4 [(12k+6 \wedge \neg(2(n-k))) \vee (2k+1 \wedge \neg 2n+1)]$, i.e. $F(8n+5,8k+4)  = F(2n+1,2k+1)$.
\item $g(8n+5,8k+5) =  (48k+30 \wedge \neg(8(n-k))) \vee (8k+5\wedge \neg 8n+5) \neq 0$, i.e. $F(8n+5,8k+5) = 0$.
\item $g(8n+7,8k) = (48k \wedge \neg(8(n-k)+7)) \vee (8k \wedge \neg 8n+7) = (48k \wedge \neg(8(n-k)) \vee (8k \wedge \neg 8n)$, i.e. $F(8n+7,8k) = F(n,k)$.
\item $g(8n+7,8k+1) = (48k+6 \wedge \neg(8(n-k)+6)) \vee (8k+1 \wedge \neg 8n+7) = (48k \wedge \neg(8(n-k))) \vee (8k \wedge \neg 8n)$, i.e. $F(8n+7,8k+1) = F(n,k)$.
\item $g(8n+7,8k+2) = (48k+12 \wedge \neg(8(n-k)+5)) \vee (8k+2 \wedge \neg 8n+7) = (48k+12 \wedge \neg(8(n-k)+4)) \vee (8k+2 \wedge \neg 8n+6)
= 2[(24k+6 \wedge \neg(4(n-k)+2))\vee (4k+1 \wedge \neg 4n+3)]$, i.e. $F(8n+7,8k+2) = F(4n+3,4k+1)$.
\item $g(8n+7,8k+3) = (48k+18 \wedge \neg(8(n-k)+4)) \vee (8k+3 \wedge \neg 8n+7) \neq 0$, i.e. $F(8n+7,8k+3) = 0$.
\item $g(8n+7,8k+4) = (48k+24 \wedge \neg(8(n-k)+3)) \vee (8k+4 \wedge \neg 8n+7) = (48k+24 \wedge \neg(8(n-k)+2)) \vee (8k+4 \wedge \neg 8n+6)
= 2[(24k+12 \wedge \neg(4(n-k)+1))\vee (4k+2 \wedge \neg 4n+3)]$, i.e. $F(8n+7,8k+2) = F(4n+3,4k+2)$.
\item $g(8n+7,8k+5) = (48k+30 \wedge \neg(8(n-k)+2)) \vee (8k+5 \wedge \neg 8n+7) \neq 0$, i.e. $F(8n+7,8k+5) = 0$.
\item $g(8n+7,8k+6) = (48k+36 \wedge \neg(8(n-k)+1)) \vee (8k+3 \wedge \neg 8n+7) \neq 0$, i.e. $F(8n+7,8k+6) = 0$.
\item $g(8n+7,8k+7) = (48k+42 \wedge \neg(8(n-k))) \vee (8k+3 \wedge \neg 8n+7) \neq 0$, i.e. $F(8n+7,8k+7) = 0$.
\item $g(4n+3,4k+3) = (24k+18 \wedge \neg(4(n-k))) \vee (4k+3 \wedge \neg 4n+3) \neq 0$, i.e. $F(4n+3,4n+3) = 0$. 
\end{itemize}
\end{proof}

\begin{theorem} \label{thm:recursivecows}
Let $a(n) =  \sum_{k=0}^n \left[\left(\begin{array}{c}n-k\\6k\end{array}\right) \left(\begin{array}{c}n\\k\end{array}\right) \mod 2\right]$. Then
$a(n)$ satisfies the equations $a(0) =1$, $a(2n) = a(n)$, $a(8n+1) = a(8n+3) = a(n)$, $a(8n+5) = a(2n+1)$, and $a(8n+7) = a(n) + a(4n+3)$.
In particular, the sequence  
$$a(n)= \{1, 1, 1, 1, 1, 1, 1, 2, 1, 1, 1, 1, 1, 1, 2, 3, 1, 1, 1, 1, 1, 1, 1, 2, 1, 1, 1, 1, 2, 2,\dots\}$$
is the run length transform of Narayana's cows sequence

 $1, 1, 1, 2, 3, 4, 6, 9, 13, 19, 28, 41, 60, 88, 129, 189, 277, 406, 595,\dots$.
\end{theorem}
\begin{proof}
This is a consequence of Lemma \ref{lem:a2n}, Corollary \ref{cor:rlt-k2}, Lemma \ref{lem:a2n_order3}, and Lemma \ref{lem:fandgcows}.
Note that by Lemma \ref{lem:fandgcows}, $a(8n+5) = a(n) + \sum_{k=0}^nF(2n+1,2k+1)$ which is equal to $a(2n+1)$ by Lemma \ref{lem:a2n}. Similarly $a(8n+7) = 2a(n) +  \sum_{k=0}^n F(4n+3,4k+1) + F(4n+3,4k+2)$ which is equal to $a(n) + a(4n+3)$.
\end{proof}

\subsection{Run length transform of $1,1, 2,2, 3,3,4,4,5,5,\dots$}
The following result shows that $a(n) =  \sum_{k=0}^n \left[ \left(\begin{array}{c}n+3k\\6k\end{array}\right) \left(\begin{array}{c}n\\k\end{array}\right) \mod 2 \right]$ is equal to the run length transform of the sequence $1,1,2,2,3,3,4,4,5,5,\dots$ (OEIS  \OEIS{A008619}).

\begin{lemma} \label{lem:fandgdouble}
For  $a_1 = 1$, $a_2 = 3$, $a_3 = 0$, $a_4 = 6$, the following relations hold for the function $F$:
\begin{itemize}
\item $F(8n+1,8k+1) = 0$,
\item $F(8n+3,8k+1) = F(n,k)$,
\item $F(8n+3,8k+2) = F(8n+3,8k+3) = 0$,
\item $F(8n+5,8k+1)  = F(8n+5,8k+5) = 0$,
\item $F(8n+5,8k+4) = F(2n+1,2k+1)$,
\item $F(8n+7,8k+i) = 0$ for $i \in \{1,3,6,7\}$,
\item $F(8n+7, 8k) = F(n,k)$,
\item $F(8n+7,8k+2) = F(4n+3,4k+1)$,
\item $F(8n+7,8k+4) =  F(4n+3,4k+2)$,
\item $F(8n+7,8k+5) = F(2n+1,2k+1)$,
\item $F(4n+3,4k+3) = 0$.

\end{itemize}
\end{lemma}
\begin{proof}
$g(8n+1,8k+1) = (48k+6 \wedge \neg(8(n+3k)+4)) \vee (8k+1 \wedge \neg 8n+1) \neq 0$, i.e
$F(8n+1,8k+1) = 0$. 

\noindent $g(8n+3,8k+1) = (48k+6 \wedge \neg(8(n+3k)+6)) \vee (8k+1 \wedge \neg 8n+3) =  (48k \wedge \neg(8(n+3k))) \vee (8k \wedge \neg 8n) = 8g(n,k)$, i.e. $F(8n+3,8k+1) = F(n,k)$.

\noindent$g(8n+3,8k+2) = (48k+12 \wedge \neg(8(n+3k)+9)) \vee (8k+2 \wedge \neg 8n+3) \neq 0$, i.e. $F(8n+3,8k+2) = 0$.

\noindent$g(8n+3,8k+3) = (48k+18 \wedge \neg(8(n+3k)+12)) \vee (8k+3 \wedge \neg 8n+3) \neq 0$, i.e. $F(8n+3,8k+3) = 0$.

\noindent$g(8n+5,8k+1) = (48k+6 \wedge \neg(8(n+3k)+8)) \vee (8k+1 \wedge \neg 8n+5) \neq 0$, i.e. $F(8n+5,8k+1) = 0$.

\noindent$g(8n+5,8k+4) = (48k+24 \wedge \neg(8(n+3k)+17)) \vee (8k+4 \wedge \neg 8n+5) =  (48k+24 \wedge \neg(8(n+3k)+16)) \vee (8k+4 \wedge \neg 8n+4) = 
4 [(12k+6 \wedge \neg(2n+6k+4)) \vee (2k+1 \wedge \neg 2n+1)] = 4g(2n+1,2k+1)$, i.e. $F(8n+5,8k+4)  = F(2n+1,2k+1)$.

\noindent$g(8n+5,8k+5) =  (48k+30 \wedge \neg(8(n+k)+20)) \vee (8k+5\wedge \neg 8n+5) \neq 0$, i.e. $F(8n+5,8k+5) = 0$.

\noindent$g(8n+7,8k) =  (48k \wedge \neg(8(n+3k)+7)) \vee (8k\wedge \neg 8n+7)  =  (48k \wedge \neg(8(n+3k)) \vee (8k\wedge \neg 8n) = 8g(n,k)$, i.e. $F(8n+7,8k) = F(n,k)$.

\noindent$g(8n+7,8k+1) =  (48k+6 \wedge \neg(8(n+3k)+10)) \vee (8k+1\wedge \neg 8n+7)  \neq 0$, i.e. $F(8n+7,8k+1) = 0$.

\noindent$g(8n+7,8k+2) =  (48k+12 \wedge \neg(8(n+3k)+13)) \vee (8k+2\wedge \neg 8n+7)  =  (48k+12 \wedge \neg(8(n+3k)+12)) \vee (8k+2\wedge \neg 8n+6) = 
 2[(24k+6) \wedge \neg(4(n+3k)+6) \vee (4k+1\wedge \neg 4n+3)] = 2g(4n+3,4k+1)$, i.e. $F(8n+7,8k+2) = F(4n+3,4k+1)$.

\noindent$g(8n+7,8k+3) =  (48k+18 \wedge \neg(8(n+3k)+16)) \vee (8k+3\wedge \neg 8n+7)  \neq 0$, i.e. $F(8n+7,8k+3) = 0$.

\noindent$g(8n+7,8k+4) =  (48k+24 \wedge \neg(8(n+3k)+19)) \vee (8k+4\wedge \neg 8n+7)  =  (48k+24 \wedge \neg(8(n+3k)+18)) \vee (8k+4\wedge \neg 8n+6) = 
2[ (24k+12) \wedge \neg(4(n+3k)+9) \vee (4k+2\wedge \neg 4n+3)] = 2g(4n+3,4k+2) $, i.e. $F(8n+7,8k+4) = F(4n+3,4k+2)$.

\noindent$g(8n+7,8k+5) =  (48k+30 \wedge \neg(8(n+3k)+22)) \vee (8k+5\wedge \neg 8n+7)  =  (48k+24 \wedge \neg(8(n+3k)+16)) \vee (8k+4\wedge \neg 8n+4) = 
 4[ (12k+6) \wedge \neg(2(n+3k)+4) \vee (2k+1\wedge \neg 2n+1)] = 4g(2n+1,2k+1)$, i.e. $F(8n+7,8k+5) = F(2n+1,2k+1)$.

\noindent$g(8n+7,8k+6) =  (48k+36 \wedge \neg(8(n+3k)+25)) \vee (8k+6\wedge \neg 8n+7) \neq 0$, i.e. $F(8n+7,8k+6) = 0$.

\noindent$g(8n+7,8k+7) =  (48k+42 \wedge \neg(8(n+3k)+28)) \vee (8k+7\wedge \neg 8n+7)  \neq 0$, i.e. $F(8n+7,8k+7) = 0$.

\noindent$g(4n+3,4k+3) =  (24k+18 \wedge \neg(4(n+3k)+12)) \vee (4k+3\wedge \neg 4n+3)    \neq 0$, i.e. $F(4n+3,4k+3) = 0$.

\end{proof}

\begin{theorem} \label{thm:recursivedouble}
Let $a(n) =  \sum_{k=0}^n \left[\left(\begin{array}{c}n+3k\\6k\end{array}\right) \left(\begin{array}{c}n\\k\end{array}\right) \mod 2\right]$. Then
$a(n)$ satisfies the equations $a(0) =1$, $a(2n) = a(n)$, $a(8n+1) = a(n)$,  $a(8n+3) = 2a(n)$, $a(8n+5) = a(2n+1)$, and $a(8n+7) = a(4n+3) + a(2n+1)-a(n)$.
In particular,  the sequence $$a(n)= \{1, 1, 1, 2, 1, 1, 2, 2, 1, 1, 1, 2, 2, 2, 2, 3, 1, 1, 1, 2, 1, 1, 2, 2, 2, 2, 2, 4, 2, 2,\dots\}$$ is the run length transform of the sequence
 $1,1,2,2,3,3,4,4,5,5,\dots$.
\end{theorem}
\begin{proof}
This is a consequence of Lemma \ref{lem:fandgcows},  Corollary \ref{cor:rlt-k2}, Lemma \ref{lem:a2n_order3}, Lemma \ref{lem:fandgdouble}. By  Corollary \ref{cor:rlt-k2}, the recurrence relations implies that $a(n)$ is the run length transform of a sequence $S_n$ that satisfies $S_0 = S_1 = 1$, $S_2 = 2$ and $S_{n+1} = S_n + S_{n-1} - S_{n-2}$.
\end{proof}

\section{Fourth order recurrences}
We first start with the following Corollary to Theorem \ref{thm:rlt} for $k=3$:

\begin{corollary}\label{cor:rlt-k3}
Let $\{S_n\}_{n\geq 0}$ be defined by the recurrence $S_{n+1} = d_0S_n + d_1S_{n-1} + d_2S_{n-2} + d_3S_{n-3}$ with initial conditions $S_0 = c_0$, $S_1 = c_1$, $S_2 = c_2$, and $S_3 = c_3$. Then
the run length transform of $\{S_n\}$ is given by $\{T_n\}_{n\geq 0}$ satisfying 
\begin{itemize}
\item $T_0 = c_0$,
\item $T_{2n} = T_n$, 
\item $T_{16n+1} = c_1T_n$,  
\item $T_{16n+3} = c_2T_n$,
\item $T_{16n+5} = c_1^2 T_{n}$, 
\item $T_{16n+7} = c_3T_n$,
\item $T_{16n+9} = c_1 T_{2n+1}$,
\item $T_{16n+11} = c_2 T_{2n+1}$,
\item $T_{16n+13} = c_1 T_{4n+3}$,
\item $T_{16n+15} = d_0T_{8n+7} + d_1T_{4n+3}+d_2T_{2n+1}+d_3T_n$.
\end{itemize}
\end{corollary}

\begin{theorem} \label{thm:4thorder}
The following relations hold for the function $F$ as defined in Definition \ref{def:F}:
\begin{itemize}
\item $F(16n,16k+i) = 0$ for $i=1,\dots , 15$,
\item $F(16n+1,16k+i) = 0$ for $i =2,\dots ,15$,
\item $F(16n+3,16k+i) = 0$ for $i = 4,\dots ,15$,
\item $F(16n+5,16k+i) = 0$ for $i = 2,3,6,7,8,\dots ,15$,
\item $F(16n+7,16k+i) = 0$ for $i =8,\dots ,15$,
\item $F(16n+9,16k+i) = 0$ for $i = 2,\dots ,7$ and $i=10,\dots ,15$,
\item $F(16n+11,16k+i) = 0$ for $i = 4,5,6,7,12,13,14,15$,
\item $F(16n+13,16k+i) = 0$ for $i = 2,3,6,7,10,11,14,15$,
\item If  $ia_3 \wedge \neg ia_1 \equiv 0 \mod 16$ and $0\leq ia_1, ia_3 < 16$, then  $F(16n+i,16k) = F(n,k)$ for $i \in \{1,3,5,7,9,11,13\}$,
\item If  $ia_3 \wedge \neg ia_1 \not\equiv 0 \mod 16$, then  $F(16n+i,16k) = 0$, for $i \in \{1,3,5,7,9,11,13\}$.
\end{itemize}
\end{theorem}
\begin{proof}
The proof is similar to the proof of Theorem \ref{thm:fandg}.
\end{proof}

\subsection{Run length transform of $1,1,2,1,3,4,7,11,18,\dots$}
Consider the sequence $1$,$1$,$2$,$1$,$3$,$4$,$7$,$11$,$18$,$\dots$ which is equal to 
the coefficients in the expansion of  $\frac{1-2x^3}{1-x-x^2}$. This sequence (OEIS \OEIS{A329723}) is also equal to the sequence formed by prepending the Lucas numbers (OEIS \OEIS{A000032}) with the terms $1$, $1$.
The following result shows that $a(n) =  \sum_{k=0}^n \left[\left(\begin{array}{c}n+2k\\2n-k\end{array}\right) \left(\begin{array}{c}n\\k\end{array}\right) \mod 2\right]$ is equal to the run length transform of this sequence.

\begin{lemma} \label{lem:fandglucas}
For  $a_1 = 1$, $a_2 = 2$, $a_3 = 2$, $a_4 = -1$, the following relations hold for the function $F$:
\begin{itemize}
\item $F(16n+1,16k+1) = F(16n+3,16k+1) = F(16n+3,16k+2) = F(16n+5,16n+5)  = F(16n+7,16k+4)= F(n,k)$,
\item $F(16n+3,16k+3) = F(16n+5,16k+1) = F(16n+5,16k+4) =  0$,
\item $F(16n+7,16k+1) = F(16n+7,16k+2) = F(16n+7,16k+3) =  0$,
\item $F(16n+7,16k+5) = F(16n+7,16k+6) = F(16n+7,16k+7) =  0$,
\item $F(16n+9,16k+8) = F(16n+11,16k+3) = F(16n+11,16k+8) =  F(16n+11,16k+11) = 0$,
\item $F(16n+13,16k+4) = F(16n+13,16k+8) = F(16n+13,16k+12) =  F(16n+13,16k+13) = 0$,
\item $F(16n+9,16k+1) = F(16n+11,16k+1) =  F(16n+11,16k+1) = F(16n+13,16k+1) =  F(2n+1,2k)$,
\item $F(16n+9,16k+9) = F(16n+11,16k+9) =   F(16n+11,16k+10) =  F(2n+1,2k+1)$,
\item $F(16n+13,16k+5) = F(16n+13,16k+5) = F(4n+3,4k+1)$,
\item $F(16n+15,16k+i) = 0$ for $i = 1,2,3,4,7,9,10,11,15$,
\item $F(16n+15,16k+i) =  F(4n+3,4k+1)$ for $i=5,6,8,12,13,14$,
\item $F(16n+15,16k) = F(8n+7,8k) = F(2n+1,2k)$,
\item $F(8n+7,8k+i) = 0$ for $i = 1,2,3,7$,
\item $F(8n+7,8k+i) = F(4n+3,4k+1)$ for $i = 4,5,6$,
\item $F(4n+3,4k+1) = F(4n+3,4k+2)$,
\item $F(4n+3,4k) = F(2n+1,2k)$,
\item $F(4n+3,4k+3) = 0$,
\item $F(16n+i,16k) = 0$ for $i=1,3,5,7,9,11,13$.

\end{itemize}
\end{lemma}
\begin{proof}
First note that $2i \wedge \neg i \not\equiv 0 \mod 16$ for $i = 1,3,5,7,9,11,13$ which implies that $F(16n+i,16k) = 0$ for $i=1,3,5,7,9,11,13$ by Theorem \ref{thm:4thorder}.

$g(16n+1,16k+1) = (16(2n-k)+1 \wedge \neg(16(n+2k)+3)) \vee (16k+1 \wedge \neg 16n+1) =  (16(2n-k) \wedge \neg(16(n+2k))) \vee (16k \wedge \neg 16n)=16g(n,k)$, i.e
$F(8n+1,8k+1) = F(n,k)$. 

\noindent $g(16n+3,16k+1) = (16(2n-k)+5 \wedge \neg(16(n+2k)+5)) \vee (16k+1 \wedge \neg 16n+3) =  (16(2n-k) \wedge \neg(16(n+2k))) \vee (16k \wedge \neg 16n) = 16g(n,k)$, i.e
$F(8n+3,8k+1) = F(n,k)$. 

\noindent $g(16n+9,16k+1) = (16(2n-k)+17 \wedge \neg(16(n+2k)+11)) \vee (16k+1 \wedge \neg 16n+9) =  (16(2n-k)+16 \wedge \neg(16(n+2k)+8)) \vee (16k \wedge \neg 16n+8)= 8g(2n+1,2k)$, i.e
$F(16n+9,16k+1) = F(2n+1,2k)$.

\noindent $g(16n+9,16k+9) = (16(2n-k)+9 \wedge \neg(16(n+2k)+27)) \vee (16k+9 \wedge \neg 16n+9) =  (16(2n-k)+8 \wedge \neg(16(n+2k)+24)) \vee (16k +8 \wedge \neg 16n+8) = 8g(2n+1,2k+1)$, i.e
$F(16n+9,16k+9) = F(2n+1,2k+1)$. 

\noindent $g(16n+11,16k+1) = (16(2n-k)+21 \wedge \neg(16(n+2k)+13)) \vee (16k+1 \wedge \neg 16n+11) =  (16(2n-k)+16 \wedge \neg(16(n+2k)+8)) \vee (16k \wedge \neg 16n+8) = 8g(2n+1,2k)$, i.e
$F(16n+9,16k+1) = F(2n+1,2k)$.

\noindent $g(16n+11,16k+2) = (16(2n-k)+20 \wedge \neg(16(n+2k)+15)) \vee (16k+2 \wedge \neg 16n+11) =  (16(2n-k)+16 \wedge \neg(16(n+2k)+8)) \vee (16k \wedge \neg 16n+8) = 8g(2n+1,2k)$, i.e
$F(16n+9,16k+2) = F(2n+1,2k)$.

\noindent $g(16n+11,16k+9) = (16(2n-k)+13 \wedge \neg(16(n+2k)+29)) \vee (16k+9 \wedge \neg 16n+11) =  (16(2n-k)+8 \wedge \neg(16(n+2k)+24)) \vee (16k +8  \wedge \neg 16n+8) = 8g(2n+1,2k+1)$, i.e
$F(16n+11,16k+9) = F(2n+1,2k+1)$.

\noindent $g(16n+11,16k+10) = (16(2n-k)+12 \wedge \neg(16(n+2k)+31)) \vee (16k+10 \wedge \neg 16n+11) =  (16(2n-k)+8 \wedge \neg(16(n+2k)+24)) \vee (16k +8  \wedge \neg 16n+8) = 8g(2n+1,2k+1)$, i.e
$F(16n+11,16k+10) = F(2n+1,2k+1)$.

\noindent $g(16n+13,16k+1) = (16(2n-k)+25 \wedge \neg(16(n+2k)+15)) \vee (16k+1 \wedge \neg 16n+13) =  (16(2n-k)+16 \wedge \neg(16(n+2k)+8)) \vee (16k \wedge \neg 16n+8) = 8g(2n+1,2k)$, i.e
$F(16n+13,16k+1) = F(2n+1,2k)$.

\noindent $g(16n+13,16k+5) = (16(2n-k)+21 \wedge \neg(16(n+2k)+23)) \vee (16k+5 \wedge \neg 16n+13) =  (16(2n-k)+20 \wedge \neg(16(n+2k)+20)) \vee (16k+4  \wedge \neg 16n+12) = 4g(4n+3,4k+1)$, i.e
$F(16n+13,16k+5) = F(4n+3,4k+1)$.

\noindent $g(16n+13,16k+9) = (16(2n-k)+17 \wedge \neg(16(n+2k)+31)) \vee (16k+9 \wedge \neg 16n+13) =  (16(2n-k)+20 \wedge \neg(16(n+2k)+20)) \vee (16k+4  \wedge \neg 16n+12) = 4g(4n+3,4k+1)$, i.e
$F(16n+13,16k+9) = F(4n+3,4k+1)$.

\noindent $g(16n+15,16k+5) = (16(2n-k)+25 \wedge \neg(16(n+2k)+25)) \vee (16k+5 \wedge \neg 16n+15) =  (16(2n-k)+20 \wedge \neg(16(n+2k)+20)) \vee (16k+4  \wedge \neg 16n+12) = 4g(4n+3,4k+1)$, i.e
$F(16n+15,16k+5) = F(4n+3,4k+1)$. Similarly, $F(16n+15,16k+6) = F(16n+15,16k+8) = F(4n+3,4k+1)$.

\noindent $g(16n+15,16k) = (16(2n-k)+30 \wedge \neg(16(n+2k)+15)) \vee (16k \wedge \neg 16n+15) =  (16(2n-k)+16 \wedge \neg(16(n+2k)+8)) \vee (16k \wedge \neg 16n+8)= 8g(2n+1,2k)$, i.e
$F(16n+15,16k) = F(2n+1,2k)$. 

\noindent $g(8n+7,8k) = (8(2n-k)+14 \wedge \neg(8(n+2k)+7)) \vee (8k \wedge \neg 8n+7) =  (8(2n-k)+8 \wedge \neg(8(n+2k)+4)) \vee (8k \wedge \neg 8n+4) = 4g(2n+1,2k)$, i.e
$F(8n+7,8k) = F(2n+1,2k)$.

\noindent $g(16n+15,16k+12) = (16(2n-k)+18 \wedge \neg(16(n+2k)+39)) \vee (16k+12 \wedge \neg 16n+15) =  (16(2n-k)+20 \wedge \neg(16(n+2k)+20)) \vee (16k+4  \wedge \neg 16n+12)= 4g(4n+3,4k+1)$, i.e
$F(16n+15,16k+12) = F(4n+3,4k+1)$. Similarly, $F(16n+15,16k+13) = F(16n+15,16k+14) = F(4n+3,4k+1)$.

\noindent $g(4n+3,4k) = (4(2n-k)+6 \wedge \neg(4(n+2k)+3)) \vee (4k \wedge \neg 4n+3) =  (4(2n-k)+4 \wedge \neg(4(n+2k)+2)) \vee (4k \wedge \neg 4n+2) = 2g(2n+1,2k)$, i.e
$F(4n+3,4k) = F(2n+1,2k)$.

\noindent $g(4n+3,4k+1) = (4(2n-k)+5 \wedge \neg(4(n+2k)+5)) \vee (4k+1 \wedge \neg 4n+3) =  (4(2n-k)+4 \wedge \neg(4(n+2k)+7)) \vee (4k+2 \wedge \neg 4n+3) = g(4n+3,4k+2) $, i.e
$F(4n+3,4k+1) = F(4n+3,4k+2)$.

Similar arguments show that $F(8n+7,8k+i) = F(4n+3,4k+1)$ for $i=4,5,6$.
The rest of the proof is similar to Lemma \ref{lem:fandgdouble}.

\end{proof}

\begin{theorem} \label{thm:recursivelucas}
Let $a(n) =  \sum_{k=0}^n \left[\left(\begin{array}{c}n+2k\\2n-k\end{array}\right) \left(\begin{array}{c}n\\k\end{array}\right) \mod 2\right]$. Then
$a(n)$ satisfies the equations $a(0) =1$, $a(2n) = a(n)$, $a(16n+1) = a(n)$,  $a(16n+3) = 2a(n)$, $a(16n+5) = a(n)$, and $a(16n+7) = a(n)$.
Furthermore, $a(16n+9) = 2a(2n+1)$,
$a(16n+11) = 2a(2n+1)$, and $a(16n+13) = a(4n+3)$.
Finally $a(16n+15) = a(8n+7) + a(4n+3)$.
This implies that  the sequence $$a(n)= \{1, 1, 1, 2, 1, 1, 2, 1, 1, 1, 1, 2, 2, 2, 1, 3, 1, 1, 1, 2, 1, 1, 2, 1, 2, 2, 2, 4, 1, 1,\dots\}$$ is the run length transform of the extended Lucas sequence
 $1,1,2,1,3,4,7,11,18,\dots$.
\end{theorem}
\begin{proof}
This follows as a consequence of Lemma \ref{lem:fandgcows},  Corollary \ref{cor:rlt-k3}, Theorem \ref{thm:4thorder}, Lemma \ref{lem:a2n_order3}, Lemma \ref{lem:fandgdouble} and Lemma \ref{lem:fandglucas}. By  Corollary \ref{cor:rlt-k3}, the recurrence relations implies that $a(n)$ is the run length transform of a sequence $S_n$ that satisfies $S_0 = S_1 = S_3 = 1$, $S_2 = 2$ and $S_{n+1} = S_n + S_{n-1}$ for $n\geq 3$.
\end{proof}

\section{Conclusions}

The sequences considered in the above sections and their run length transforms are summarized in Table \ref{tbl:seq-rlt}. In the table, the coefficients $a_i$ describe the run length transform expressed as $\sum_{k=0}^n \left[\left(\begin{array}{c}a_1n+a_2k\\a_3n+a_4k\end{array}\right) \left(\begin{array}{c}n\\k\end{array}\right) \mod 2\right]$.

\begin{table}[htbp]
\begin{center}
\begin{tabular}{|l|l|l|l|l|}
\hline
{\bf \shortstack{Sequence\\description}}  & {\bf\shortstack{OEIS\\index}}  & {\bf Sequence terms} & {\bf \shortstack{Coefficients\\$(a_1,a_2,a_3,a_4)$\\of Run\\Length\\Transform}} & {\bf \shortstack{OEIS\\index\\of Run\\Length\\Transform}}\\
\hline\hline
\shortstack{Positive\\powers\\of $2$} &  \OEIS{A000079} & $1,2,4,8,\dots$ & $(1,0,0,1)$ & \OEIS{A001316} \\
\hline
\shortstack{Fibonacci\\sequence} & \OEIS{A000045} & $1, 1, 2, 3, 5, 8,\dots$ & $(1,-1,0,2)$ & \OEIS{A246028} \\
\hline
\shortstack{Truncated\\Fibonacci\\sequence} & & $1, 2, 3, 5, 8,13,\dots$ & $(0,3,0,1)$ & \OEIS{A245564} \\
\hline
\shortstack{1 plus the\\positive\\powers of $2$} & \OEIS{A011782} & $1,1,2,4,8,16, \dots$ & $(1,0,0,2)$ & \OEIS{A245195}  \\
\hline
\shortstack{1 followed\\by $2$'s} & \OEIS{A040000} & $1,2,2,2,2,2, \dots$ & $(1,2,0,2)$ & \OEIS{A277561}  \\
\hline
\shortstack{Positive\\integers} & \OEIS{A000027} & $1,2,3,4,5,6, \dots$ & $(1,1,1,-1)$ & \OEIS{A106737} \\
\hline
\shortstack{A sequence\\of 1's} & \OEIS{A000012} & $1,1,1,1,1,1, \dots$ & $(1,-1,0,1)$ & \OEIS{A000012} \\
\hline
\shortstack{Narayana's\\cows\\sequence} & \OEIS{A000930} & $1,1,1,2,3,4,6,9, \dots$ & $(1,-1,0,6)$ & \OEIS{A329720}  \\
\hline
\shortstack{Positive\\integers\\repeated} & \OEIS{A008619} & $1,1,2,2,3,3,4,4, \dots$ & $(1,3,0,6)$ &  \OEIS{A278161}\\
\hline
\shortstack{Lucas\\sequence\\prepended\\with 1,1} & \OEIS{A329723} & $1,1,2,1,3,4,7,11, \dots$ & $(1,2,2,-1)$ &  \OEIS{A329722}\\
\hline
\end{tabular}
\caption{Table of various sequences and their run length transforms expressed as sums of products of binomial coefficients mod $2$. }
\label{tbl:seq-rlt}
\end{center}
\end{table}

The run length transform has been useful in analyzing the number of ON cells in a cellular automata after $n$ iterations \cite{sloane:CA:2018}.  We show here that the run length transform can also characterize sums of products of binomial coefficients mod 2. In particular, the run length transforms of several well-known sequences correspond to sums of products of binomial coefficients mod 2.   Given the fact that several cellular automata can generate Sierpi\'{n}ski's  triangle \cite{mathworld:sierpinski} which is equivalent to Pascal's triangle mod 2, this is not surprising and suggests that there is a close relationship between cellular automata and functions of binomial coefficients mod 2.

\end{document}